\documentclass[a4paper,leqno,11pt]{amsart}
\usepackage{mathrsfs}
\usepackage{amsfonts}
\usepackage{amssymb}
\usepackage{amsxtra}
\usepackage{dsfont}
\usepackage{mathtools}
\usepackage{enumitem}
\usepackage{bm}
\usepackage{mathtools}
\usepackage{yhmath}

\usepackage{graphicx}
\usepackage[english]{babel}

\usepackage{color}

\usepackage[hidelinks]{hyperref}
\usepackage{nameref}
\hypersetup{colorlinks = true, urlcolor = blue, linkcolor = blue, citecolor = red}

\numberwithin{equation}{section}
\newtheorem{theorem}[equation]{Theorem}
\newtheorem{lemma}[equation]{Lemma}
\newtheorem{proposition}[equation]{Proposition}
\newtheorem{corollary}[equation]{Corollary}

\theoremstyle{definition}
\newtheorem{remark}[equation]{Remark} 

\newcommand{\RR}{\mathbb{R}}
\newcommand{\ZZ}{\mathbb{Z}}
\newcommand{\CC}{\mathbb{C}}
\newcommand{\NN}{\mathbb{N}}
\newcommand{\Z}{\mathbb{Z}}

\newcommand{\N}{\mathbb{N}}
\newcommand{\T}{\mathbb{T}}
\newcommand{\R}{\mathbb{R}}

\newcommand{\fa}{r}

\newcommand{\ind}[1]{{\mathds{1}_{{#1}}}}

\DeclareMathOperator{\Ima}{Im}
\DeclareMathOperator{\Rea}{Re}
\DeclareMathOperator{\sgn}{sgn}
\DeclareMathOperator{\Vol}{Vol}

\renewcommand{\baselinestretch}{1.2}

\begin{document}

\title[Lattice points on high-dimensional cross-polytopes]{Uniform estimates for Delannoy numbers \\ and dimension-free estimates for discrete maximal functions over cross-polytopes}

\begin{abstract}
We prove a uniform upper and lower bound for Delannoy numbers. This is achieved by using the representation of Delannoy numbers as the number of lattice points in high-dimensional cross-polytopes (also known as hyper-octahedrons or $\ell^1$~balls) and proving a uniform (dimension-free) count for these lattice points. Using this count, we establish dimension-free estimates for discrete maximal functions over cross-polytopes. By proving a comparison principle with the continuous setting, we obtain a dimension-free estimate on all $\ell^p(\Z^d)$ spaces for radii $R>C d^{3/2}.$
We also treat the full maximal function on $\ell^2(\Z^d)$ for small radii $R\le d^{1-\varepsilon}$ and the dyadic maximal function for any radii.  

\end{abstract}

\subjclass[2020]{05A16, 11B75, 42B25\vspace{0.25em}}
 \keywords{Delannoy number, discrete cross-polytopes, dimension-free estimates, discrete maximal function\vspace{0.25em}}

\thanks{This research was funded in whole or in part by National Science Centre, Poland, grant Sonata Bis 2022/46/E/ST1/00036. Jakub Niksiński was supported by the NSF CAREER grant DMS-2236493.
For the purpose of Open Access the authors have applied a CC BY public copyright license to any Author Accepted Manuscript (AAM) version arising from this submission.}

\author{Dariusz Kosz}
\address{\scriptsize Dariusz Kosz (\textnormal{dariusz.kosz@pwr.edu.pl}) \newline
	Faculty of Pure and Applied Mathematics, Wroc{\l}aw University of Science and Technology, 50-370 Wroc{\l}aw, Poland
    \vspace{-0.25em} 
}

\author{Jakub Niksi\'nski}
\address{\scriptsize Jakub Niksiński (\textnormal{trolek1130@gmail.com})
\newline  Department of Mathematics, Rutgers University, Piscataway, NJ 08854, USA
     \vspace{-0.25em}
}

\author{B{\l}a{\.z}ej Wr{\'o}bel}
\address{\scriptsize B{\l}a{\.z}ej Wr{\'o}bel (\textnormal{blazej.wrobel@math.uni.wroc.pl}) \newline
	Institute of Mathematics of the Polish Academy of Sciences, 00-656 Warsaw, Poland \newline 
	Institute of Mathematics, University of Wroc{\l}aw, 50-384 Wroc{\l}aw, Poland
     \vspace{-0.25em}
}

\maketitle		
 
\section{Introduction} \label{S1}

\subsection{Statement of the results}
\label{S1 stat}

 A Delannoy number $D(d,n)$, where $d, n$ are nonnegative integers, counts the paths from the southwest corner $(0, 0)$ of a rectangular grid to the northeast corner $(d, n)$, using only single steps north, northeast, or east. Clearly, $D(d,n) = 1$ if $\min \{d,n\} = 0$. Otherwise, for $d,n \in \NN$, there are several explicit formulas for these numbers in terms of sums involving binomial coefficients. For instance, using
\begin{equation}
\label{eq: Delsum}
 D(d,n)=\sum_{k=0}^{\min(d,n)} 2^k \binom{d}{k} \binom{n}{k},
 \end{equation}
one can easily prove that
\begin{equation}
\label{eq: Del l1ball}
 D(d,n)=|B_n^d\cap\Z^d|;
\end{equation}
see, e.g. \cite[Lemma~2.2]{Ni1}. The symbol $|B_n^d\cap\Z^d|$ above denotes the number of lattice points in the closed $d$-dimensional $\ell^1$~ball of radius $n$ (also known as cross-polytope, hyper-octahedron, or orthoplex), that is,
\[
B^d_n \coloneqq \big \{ x = (x_1, \dots, x_d) \in \RR^d : |x_1| + \dots + |x_d| \leq n \big \}. 
\]
We remark that \eqref{eq: Delsum} implies that 
\begin{equation}
\label{eq: B change n and d}
|B_n^d\cap\Z^d|=D(d,n)=D(n,d)=|B_d^n\cap\Z^n|,
\end{equation}
which -- apart from being quite puzzling -- is important for our method, as it allows swapping $n$ and $d$ in the lattice point count.

Looking at \eqref{eq: Delsum}, it is clear that for $n$ much larger than $d$ the dominant term is $2^d \binom{n}{d}$, and it is of the order of the Lebesgue measure $\Vol(B_n^d).$ On the other hand, if $n$ is much smaller than $d$, then the dominant term in \eqref{eq: Delsum} is $2^n \binom{d}{n}$, and it is equal to the number of lattice points in $S^d_n\cap \{-1,0,1\}^d,$ where
 $S_n^d$ denotes the boundary of the cross-polytope $B_n^d$, that is,
\[
S^d_n \coloneqq \big \{ x = (x_1, \dots, x_d) \in \RR^d : |x_1| + \dots + |x_d| =n \big \}. 
\]
However, it is not clear what ``much larger'' or ``much smaller'' means exactly, and how and when the transition between these two behaviors occurs. This problem is studied in our paper, especially in the context of proving dimension-free estimates for discrete maximal functions over $B_n^d.$

The first purpose of this paper is to give a uniform estimate from above and below for $D(d,n)$. This will be achieved by using \eqref{eq: Del l1ball} and \eqref{eq: B change n and d}, and by establishing such a uniform lattice point count for $B_n^d\cap\Z^d$ and $S_n^d\cap\Z^d$ when $n\le d.$  In what follows, for two quantities $X$ and $Y$ we write $X\lesssim Y$ when $X \le C Y,$ where $C$ is a universal (absolute) constant. We also write $X\approx  Y$ when $X\lesssim Y$ and $Y\lesssim X$ hold simultaneously. Our main lattice point count below is given in terms of the two quantities
\[
\alpha \coloneqq \frac{n}{d}, \quad r \coloneqq \frac{\sqrt{1+\alpha^2}-1}{\alpha}.
\]
Notice that for $n\le d$ we have $r\le 1/2$ and $r\approx \alpha.$

\begin{theorem}
\label{thm:lat:extract} 
For $d,n \in \N$ satisfying  $n \leq d$ we have 
\begin{equation}
\label{eq: B_n small univ} 
 D(d,n)=|B_n^d\cap\Z^d| \approx  |S_n^d\cap\Z^d|
 \approx  \bigg(\frac{1+r}{1-r}\bigg)^{d} r^{-n} \frac{1}{\sqrt{n}}.
\end{equation}
In particular, rewriting \eqref{eq: B_n small univ} in terms of $d$ and $n$, we obtain 
\[
D(d,n) \approx \bigg(\frac{d}{\sqrt{d^2+n^2}-n}\bigg)^{d} \bigg( \frac{n}{\sqrt{d^2+n^2}-d} \bigg)^n \frac{1}{\sqrt{n}}.
\]
Furthermore, there exists a function 
\[b(z)=\frac{z^2}{12} - \frac{3z^4}{160} + \sum_{k=3}^{\infty} b_k z^{2k},\]
which is holomorphic inside the disk $\{ z \in \CC : |z|<2/3 \}$ and for which
\begin{align}
\label{eq: B_n small}
|B_n^d\cap\Z^d| &\approx |S_n^d \cap \Z^d| \approx  (2e/\alpha)^n \frac{1}{\sqrt{n}
}\exp ( nb(\alpha) ),\qquad n\le d/2,
\\
\label{eq: B_n large}
|B_n^d\cap\Z^d| &\approx (2e \alpha)^d \frac{1}{\sqrt{d}
}\exp ( d b(\alpha^{-1}) ),\qquad n\ge 2d.
\end{align}

\end{theorem}

Rewriting \eqref{eq: B_n small} in terms of binomial coefficients, we obtain the following corollary for $n\le d/2.$
\begin{corollary} \label{cor:lat:compa small}
For $d,n \in \N$ satisfying  $n \leq d/2$ we have 
\begin{equation*}
|B_n^d\cap\Z^d| \approx |S^d_n\cap\Z^d| \approx 2^n \binom{d}{n} \exp \Big(\frac{n \alpha}{2}+ \frac{n \alpha^2}{4} + O (n \alpha^3) \Big).
\end{equation*}
In particular, if $n \lesssim d^{1/2}$, then
\begin{equation*}
|B^d_n\cap\Z^d| \approx |S^d_n\cap\Z^d| \approx 2^n \binom{d}{n}=|S^d_n\cap \{-1,0,1\}^d|.
\end{equation*}
\end{corollary}

Rewriting \eqref{eq: B_n large} in terms of the Lebesgue measure of the cross-polytope $\Vol (B^d_n)$, we obtain the following corollary for $n\ge 2d$.
\begin{corollary} \label{cor:lat:compa large}
For $d,n \in \N$ satisfying  $n \ge 2d$ we have 
\begin{equation*}
|B^d_n\cap\Z^d| \approx \Vol (B^d_n) \exp \Big(  \frac{d}{12 \alpha^2} + O(d/\alpha^3) \Big).
\end{equation*}
In particular, if $n \gtrsim d^{3/2}$, then
\begin{equation}
\label{eq: compa dis con}
|B^d_n\cap\Z^d|\approx \Vol (B^d_n).
\end{equation}
\end{corollary}

\begin{remark}
    \label{rem: uni asym} The uniform estimates meant by the symbol $\approx$ appearing in Theorem~\ref{thm:lat:extract} and Corollaries~\ref{cor:lat:compa small}~and~\ref{cor:lat:compa large} are sufficient for our purposes. Therefore, we do not pursue more refined asymptotics.
\end{remark}

\begin{remark}
    \label{rem: Ehrart} Using the theory of Ehrhart polynomials, one can prove 
\begin{equation}
\label{eq: compa dis con sharp}
|B^d_n\cap\Z^d|\ge  \Vol (B^d_n)
\end{equation}
for all $d,n \in \NN$, which can be viewed as a one-sided strengthening of \eqref{eq: compa dis con}.
For a $d$-dimensional polytope $P$ with integer vertices its Ehrhart polynomial is defined by $i(P,n)=|nP \cap \ZZ^d|.$ Ehrhart~\cite{Ehr} showed that $i(P,n)$ is indeed a polynomial $\sum_{k=0}^d c_k(P) n^k$ of degree $d$ in $n$, whose leading coefficient $c_d(P)$ is equal to the Lebesgue measure $\Vol(P)$. We refer to \cite[Chapter 11.3]{Rob} for a comprehensive description. In our case, $i(B_1^d,n)$ is clearly a polynomial, since \eqref{eq: Delsum} can be rewritten as
    $D(d,n)=i(B_1^d,n)=\sum_{k=0}^{d} 2^k \binom{d}{k} \binom{n}{k}$,
where we use the convention that $\binom{n}{k}=0$ if $k>n.$ An Ehrhart polynomial $i(P,n)$ is called Ehrhart positive if all of its coefficients $c_k(P)$ are positive; see \cite{Liu2019}. Liu verified that $i(B_1^d,n)$ is Ehrhart positive; see \cite[Section~2.2.1]{Liu2019}. This property is not clear from \eqref{eq: Delsum} as this expression is not of the form $\sum_{k=0}^{d} c_k(B_1^d) n^k.$
Now, knowing that $i(B_1^d,n)$ is Ehrhart positive and using the fact that $c_d(B_1^d)=\Vol(B_1^d)$, we easily obtain \eqref{eq: compa dis con sharp}.
\end{remark}

The second, and the main, purpose of this paper is to prove dimension-free estimates for discrete maximal functions over the cross-polytopes. The lattice point count provided in Theorem~\ref{thm:lat:extract} will be an important ingredient in all these dimension-free estimates. Let $f \colon \ZZ^d \to \RR$. We set
\[
M_R^d f(x) \coloneqq \frac{1}{|B_R^d \cap \ZZ^d|} \sum_{y \in B_R^d \cap \ZZ^d} f(x-y), \qquad x \in \ZZ^d,
\]
for every $d \in \NN$ and $R \in [0,\infty)$, where $B_0^d$ is the singleton $\{0\}$ in $\RR^d$. Note that, since $B_R^d \cap \ZZ^d=B_{\lfloor R \rfloor}^d \cap \ZZ^d$, we may restrict to nonnegative integers $R$. For any $E\subseteq [0,\infty)$ we define the associated maximal function $M_{*,E}^d f$ by  
\[
M_{*,E}^d f(x) \coloneqq
\sup_{R \in E} |M_R^d f(x)|, \qquad x \in \ZZ^d,
\]
and we use the abbreviated notation $M_{*}^d f$ if $E = [0, \infty)$. We also consider the corresponding spherical (boundary) averages
\[
\mathcal S_R^d f(x) \coloneqq \frac{1}{| S_R^d \cap \ZZ^d|} \sum_{y \in \mathcal S_R^d \cap \ZZ^d} f(x-y), \qquad x \in \ZZ^d,
\]
where again $S_0^d$ is the singleton $\{0\}$ in $\RR^d$, and their maximal function
\[
\mathcal S_{*,E}^d f(x) \coloneqq
\sup_{R \in E} | \mathcal S_R^d f(x)|, \qquad x \in \ZZ^d.
\]

We shall prove dimension-free results for maximal functions in three separate regimes of radii. Our first result here is a dimension-free bound for all $p \in (1,\infty)$ and large radii. The input from the lattice point count required in its proof is given in Corollary~\ref{cor:lat:compa large}.

\begin{theorem} \label{T1}
	Fix $c \in \RR_+$ and $p \in (1, \infty)$. Then there exists a constant $C(c,p) \in \RR_+$ such that for all $d \in \NN$ we have the dimension-free bound
 	\[
	\| M_{*,[cd^{3/2},\infty)}^d f \|_{\ell^p(\ZZ^d)} \leq C(c,p) \| f \|_{\ell^p(\ZZ^d)}, \qquad f\in \ell^p(\ZZ^d).
	\]
\end{theorem}
\noindent 

Our second result refers to dimension-free bounds for $p \in [2,\infty )$ and small radii. Here the crucial consequence of Theorem~\ref{thm:lat:extract} is a couple of concentration results, Theorem~\ref{thm:3.1} and Corollary~\ref{cor:3.2}, from Section~\ref{sec: full small}.

\begin{theorem} \label{T2}
	Fix $\varepsilon \in (0,1)$. Then there exists a constant $C(\varepsilon) \in \RR_+$ such that for all $d \in \NN$ and $p\in [2,\infty)$ we have the dimension-free bound
   \begin{equation}
   \label{eq: max spher small}
	\| \mathcal S_{*,(0,d^{1-\varepsilon}]}^d f \|_{\ell^p(\ZZ^d)} \leq C(\varepsilon) \| f \|_{\ell^p(\ZZ^d)},\qquad f\in \ell^p(\ZZ^d).
	\end{equation}
     Thus, for all $d \in \NN$ and $p\in [2,\infty)$ we also have the dimension-free bound
 	\begin{equation*}
	\| M_{*,(0,d^{1-\varepsilon}]}^d f \|_{\ell^p(\ZZ^d)} \leq C(\varepsilon) \| f \|_{\ell^p(\ZZ^d)},\qquad f\in \ell^p(\ZZ^d).
	\end{equation*}
\end{theorem}

Our third result concerns $p \in [2, \infty)$ and dyadic radii belonging to $\mathcal D \coloneqq \{2^n : n\in \NN\}$, and its proof relies on Theorem~\ref{thm:lat:extract} via Corollary~\ref{cor: iso}. This third result and the second part of Theorem~\ref{T2} will also be covered by more general results in a forthcoming paper by the second author~\cite{Ni3}.
\begin{theorem} \label{T3}
	There exists a universal constant $C \in \RR_+$ such that for all $d \in \NN$ and $p\in [2,\infty)$ we have the dimension-free bound
    \begin{equation*}
	\| M_{*,\mathcal D}^d f \|_{\ell^p(\ZZ^d)} \leq C \| f \|_{\ell^p(\ZZ^d)},\qquad f\in \ell^p(\ZZ^d).
	\end{equation*}
\end{theorem}

\subsection{Historical background and related results}

Delannoy numbers were introduced at the end of 19th century by the French amateur mathematician H.~Delannoy~\cite{Delannoy}; see \cite{BanderierSchwer} some historical references. Among these numbers, one distinguishes the central Delannoy numbers $D(n,n)$ for which it is well-known that
\[
D(n,n)\approx \frac{(1+\sqrt{2})^{2n}}{\sqrt{n}}.
\]
As it should, this clearly matches our formula \eqref{eq: B_n small univ} with $\alpha=1$ and $r=\sqrt{2}-1.$ For the noncentral Delannoy numbers $D(d,n)$ such that $n/d=\alpha \in (\alpha_0,\alpha_1)$, where $0<\alpha_0<\alpha_1<\infty$ are fixed constants, Pemantle and Wilson \cite[p.~140]{PW2002} obtained the following asymptotics
\begin{equation}
\label{eq: Per Will bound} 
\sqrt{\frac{1}{2\pi}} \left(\frac{1+r}{1-r}\right)^{d} r^{-n}
\sqrt{\frac{\alpha}{d\big(1+\alpha-\sqrt{1+\alpha^2}\big)^2\sqrt{1+\alpha^2}}} .
\end{equation}
We note that \eqref{eq: Per Will bound} boils down to our bound \eqref{eq: B_n small univ} for such $\alpha$, since in this case $n\approx d$ and all factors depending on $\alpha$ are of order $1$.
The analysis in \cite{PW2002} is based on a more general approach of studying multivariate generating functions of multivariate sequences; see also the book \cite{ACSVbook}. For example, $F(z,w) \coloneqq \frac{1}{1-z-w-zw}$ is the bivariate generating function of $D(d,n).$ 

Dimension-free estimates for centered Hardy--Littlewood maximal operators were first studied in the continuous
context. For each $t \in \RR_+$ let $\mathcal B_t^{d}$ denote the operator that averages over centered continuous $d$-dimensional Euclidean balls of radius $t$. 
For any locally integrable $f \colon \RR^d \to \RR$ let \[
\mathcal B_*^d f \coloneqq \sup_{t \in \R_+}|\mathcal B_t^{d}f(x)|
\]
be the corresponding maximal function. In 1982 Stein~\cite{SteinMax} (see
also Stein and Str\"omberg \cite{StStr}) proved that for every fixed $p\in(1, \infty]$ there exists a constant $C_p \in \RR_+$ independent of $d \in \NN$ such that 
\begin{align*}
\|\mathcal B_*^d f\|_{L^p(\RR^d)} \leq C_p \| f \|_{L^p(\RR^d)}.
\end{align*}
In the following years, Bourgain \cite{B1,B2,B3}, Carbery \cite{Car1}, and M\"uller \cite{Mul1} significantly extended Stein's result by considering various symmetric convex bodies $G$ in place of the Euclidean ball. From the perspective of our paper, M\"uller's work \cite{Mul1} is the most important. It implies a dimension-free bound for the continuous Hardy--Littlewood maximal operator $\mathcal M_*^d$ corresponding to averages over cross-polytopes
\begin{equation*}
\mathcal M_t^d f(x) \coloneqq \frac{1}{\Vol(B_t^d)} \int_{B_t^d} f(x-y) \, {\rm d}x, \qquad x \in \RR^d.
\end{equation*}
Namely, for every fixed $p \in (1, \infty]$ there exists a constant $\mathcal C(p) \in \RR_+$ independent of $d \in \NN$ such that
\begin{equation} \label{E0}
\| \mathcal M_{*}^d f \|_{L^p(\RR^d)} \leq \mathcal C(p) \| f \|_{L^p(\RR^d)}.
\end{equation}

The study of dimension-free inequalities for centered Hardy--Littlewood maximal functions in the discrete context was initiated in \cite{BMSW3} by the third author together with Bourgain, Mirek, and Stein, and continued in \cite{balls} and \cite{KMPW}, among others. From the perspective of our paper, the earliest article \cite{BMSW3} and the recent contributions of the second author \cite{Ni1, Ni2} are the most relevant. The comparison principle formulated in \cite[Theorem~1]{BMSW3} implies the dimension-free bound in the large scales regime $n\ge d^2,$ i.e.
\[
	\| M_{*,[d^2,\infty)}^d f \|_{\ell^p(\ZZ^d)} \leq C(p) \| f \|_{\ell^p(\ZZ^d)}
	\]
for all $p\in (1,\infty)$. In \cite{Ni1} the second author proves a dimension-free bound for the dyadic maximal function in the small scales regime $n \leq d^{1/2}$, i.e.
\[
	\| M_{*,(0,d^{1/2}]\cap \mathcal D}^d f \|_{\ell^p(\ZZ^d)} \leq C(p) \| f \|_{\ell^p(\ZZ^d)}
	\]
for all $p\in [2,\infty)$. Very recently, this result was improved in \cite{Ni2} to 
\[
	\| M_{*,(0,d]\cap \mathcal D}^d f \|_{\ell^p(\ZZ^d)} \leq C(p) \| f \|_{\ell^p(\ZZ^d)}
	\]
for all $p\in [2,\infty)$. It is plausible to believe that the full dimension-free bound 
\begin{equation}
\label{eq: conjecture}
	\| M_{*}^d f \|_{\ell^p(\ZZ^d)} \leq C(p) \| f \|_{\ell^p(\ZZ^d)}
\end{equation}
    holds for all $p\in (1,\infty)$, just as in the continuous case \eqref{E0}. In particular, Theorems~\ref{T1},~\ref{T2},~and~\ref{T3} can be seen as partial progress toward \eqref{eq: conjecture}. 
    
    We remark that, in view of \cite[Theorem~1]{KMPW}, the discrete setting is always the harder one. Specifically, the $L^p(\RR^d)$ norms of the continuous maximal operators do not exceed the $\ell^p(\ZZ^d)$ norms of their discrete analogues.
    
\subsection{Structure of the article and our methodology}

Section~\ref{sec: lat} is devoted to the proof of Theorem~\ref{thm:lat:extract}. Our approach is similar to the one used in \cite[Section~3]{NiWr}. Namely, we apply \eqref{eq: Del l1ball} and express the number of lattice points in $B_n^d\cap \ZZ^d$ as the complex integral over the corresponding univariate generating function of the sequence $D(d,n)$, treating $d$ as a fixed parameter; see \eqref{eq:Cauchyball}. To analyze \eqref{eq:Cauchyball} we ultimately use the saddle point method as in \cite[Section~3]{NiWr}. However, contrary to \cite[Section~3]{NiWr}, our case is more explicit, which allows us to prove \eqref{eq: B_n small univ} all the way to $n\le d.$ After completing the proof of Theorem~\ref{thm:lat:extract} in this manner, we realized that the multivariate methods from \cite[p.~140]{PW2002} might yield the same result. In that work, the authors assume that $\alpha$ is bounded away from $0$ and $\infty.$ However, it seems plausible that their methods work even when this assumption is dropped. In summary, the univariate approach proposed here is simpler and sufficient for our purposes but possibly less fruitful.

Section~\ref{sec: full large} contains the derivation of Theorem~\ref{T1}. The proof relies on transferring the bounds for $\mathcal M_{*}^d$ to the case of $M_{*,[cd^{3/2},\infty)}^d$. In view of \cite[Theorem~1]{BMSW3} it was previously known that such a transference is possible for $M_{*,[d^{2},\infty)}^d$ instead. Thus, Theorem~\ref{T1} can be interpreted as a strengthened version of \cite[Theorem~1]{BMSW3} for $\ell^1$~balls. That such a strengthened version is possible is due to the geometry of the $\ell^1$~ball, which allows us to apply the central limit theorem at a crucial point in Lemma~\ref{L1}. On the other hand, Corollary~\ref{cor:lat:compa large} hints that, from the perspective of transference arguments, the exponent $\frac{3}{2}$ may be the best possible.  This is because it seems that for transference arguments to apply, the discrete and continuous balls of the same radius should be roughly of the same size. 

In Section~\ref{sec: full small} we justify Theorem~\ref{T2}. The proof uses two consequences of Theorem~\ref{thm:lat:extract}, that is, Theorem~\ref{thm:3.1} and Corollary~\ref{cor:3.2}, as well as dimension-free bounds for a multiparameter combinatorial maximal function  \cite[Theorem~1.5]{NiWr}. With these three ingredients, the analysis is very close to that in \cite[Section~4]{NiWr}. For this reason, we decided to keep our arguments brief, mainly pointing out the differences.

Finally, Section~\ref{sec: dyad} contains the proof of Theorem~\ref{T3}. Here the argument is a repetition of the one in \cite[Section~4]{KMPW}, with Corollary~\ref{cor: iso} being the only new ingredient not present in \cite{KMPW}. As before, we omit the details and focus exclusively on the proof of Corollary~\ref{cor: iso}, which again uses Theorem~\ref{thm:lat:extract}.

\subsection{Notation} The symbols $\CC, \, \RR, \, \ZZ$ and $\T \coloneqq \RR \setminus \ZZ$ have their usual meaning. We let $\RR_+ \coloneqq (0,\infty)$, $\NN \coloneqq \{1, 2, \dots \}$, and $[n] \coloneqq \{1, \dots, n\}$ for $n \in \NN$. 

For $X,Y \in \R_+$
we write $X \lesssim_{\delta} Y$ if there is a constant
$C_{\delta} \in \RR_+$ depending on $\delta$ such that $X\le C_{\delta}Y$.
We write $X \approx_{\delta} Y$ when
$X \lesssim_{\delta} Y\lesssim_{\delta} X$. We omit the subscript
$\delta$ if the implicit constants are universal.
Similar conventions apply to big~$O$ notation. In particular, $X=O_{\delta}(Y)$ means that $|X|\lesssim_{\delta} |Y|$.

For $y \in \R$ let $e(y) \coloneqq e^{- 2\pi i y}$.
Recall that the Fourier transform $\ell^2(\ZZ^d) \ni f \mapsto \widehat f \in L^2(\T^d)$ and its inverse $\mathcal F^{-1}$ 
are isometries, and if $f \in \ell^1(\ZZ^d)$, then
\[
\widehat f (\xi) \coloneqq \sum_{x \in \ZZ^d} f(x) e(x \cdot \xi)
= \sum_{x \in \ZZ^d} f(x) e^{-2\pi i (x_1 \xi_1 + \dots + x_d \xi_d)}, \qquad \xi \in \T^d.
\]

From now on, we 
shall drop the superscript $d$ from certain symbols. In particular, we shall 
write $B_n, \, S_n, \, M_n, \, \mathcal{S}_n, \, \mathcal{M}_n, \, M_{*,E}, \,\mathcal{S}_{*,E}, \, \mathcal{M}_* $ 
in place of $B_n^d, \, S_n^d, \, M_n^d, \, \mathcal{S}_n^d, \, \mathcal{M}_n^d, \, M_{*,E}^d, \, \mathcal{S}_{*,E}^d, \, \mathcal{M}_*^d$, respectively.

\section{Lattice point count -- proof of Theorem \ref{thm:lat:extract}.}
\label{sec: lat}

In this section we prove Theorem~\ref{thm:lat:extract}. The reasoning here follows closely that presented in \cite[Section~3]{NiWr}. 

The formula \eqref{eq: B_n small univ} is expressed in terms of the auxiliary function 
\begin{equation*}
h(z) \coloneqq \sum_{k \in \Z} z^{|k|}=\frac{1+z}{1-z}
\end{equation*}
defined on the open disk $\{z\in \CC : |z|<1\}$. The crucial observation is that
\[
h(z)^d= \Big( \sum_{k \in \Z} z^{|k|} \Big)^d = \sum_{n=0}^{\infty} |S_n \cap \Z^d| z^n
\]
and 
\[
\frac{h(z)^d}{1-z}= \sum_{n=0}^\infty \Big( \sum_{k=0}^n |S_k \cap \Z^d | \Big)z^n= \sum_{n=0}^{\infty} |B_n \cap \Z^d | z^n.
\]
Thus, by Cauchy's integral formula, for any $s\in (0,1)$ we have 
\begin{equation}
\label{eq:Cauchyball}
|S_n \cap \Z^d|= \frac{1}{2 \pi i} \oint_{\Gamma_s} h(z)^d \frac{{\rm d}z}{z^{n+1}}, \quad|B_n \cap \Z^d|= \frac{1}{2 \pi i} \oint_{\Gamma_s} \frac{h(z)^d}{1-z}\frac{{\rm d}z}{z^{n+1}},
\end{equation}
where $\Gamma_s \coloneqq \{ z \in \CC : |z|=s \}$.

Recall that  
\begin{equation}
\label{eq: rdef}
\alpha=\frac{n}{d},\quad \fa=\frac{\sqrt{1+\alpha^2}-1}{\alpha}. 
\end{equation}
In view of \eqref{eq: B change n and d}, the formula \eqref{eq: B_n small} also implies \eqref{eq: B_n large} in Theorem~\ref{thm:lat:extract}. Hence, we may restrict to the case $n\le d$ and, in particular, we have $\alpha\le 1$ and
\begin{equation}
\label{eq: alpha}
\frac{\alpha}{\sqrt{2}+1}\le r\le \frac{\alpha}{2}\le \frac12.
\end{equation}
Furthermore, we note that $r$ behaves like $\alpha/2$ when $\alpha$ is small. 

We move towards the proof of Theorem~\ref{thm:lat:extract} in the case $n \leq d$. The proof is similar to the proof of \cite[Theorem~3.4]{NiWr}. Thus, we shall be brief and give more details only when there are significant differences. As in \cite[Section~3]{NiWr}, we ultimately use the saddle point method. However, in our case a number of simplifications occur and several statements can be made explicit.

\begin{lemma}
    \label{lem:3.5}
    Let $\alpha \in (0,1)$. Then $\fa$ is the unique solution of the equation
\begin{equation}
\label{eq:lem:3.5}
z\frac{h'(z)}{h(z)}=\alpha
\end{equation}
in the open unit disk $\{ z \in \CC : |z|<1 \}$. 
\end{lemma}
\begin{proof}
Since the left-hand side of \eqref{eq:lem:3.5} is equal to $\frac{2z}{1-z^2}$, our task is to solve the quadratic equation
$2z=(1-z^2)\alpha$, which has two real roots
\[
z_1 = \frac{-1 - \sqrt{1+\alpha^{2}}}{\alpha},\quad z_2=\frac{-1 + \sqrt{1+\alpha^{2}}}{\alpha}.
\]
Since $z_1<-1$ and, as in \eqref{eq: alpha}, we have $|z_2| < 1$, the proof is complete. 
\end{proof}

In what follows, we consider the holomorphic function
   \begin{equation}
   \label{eq:fdefi}
    f(z) \coloneqq \log(h(z))- \alpha \log(z)=\log(1+z)-\log(1-z)-\alpha \log(z)
    \end{equation}
    defined on a sufficiently small neighborhood of the set
    \[
    \Gamma_r^{\leq \delta} \coloneqq \big\{z \in \mathbb{C} : |z|=r \text{ and } |\arg(z)| \leq \delta \big\},
    \]
    where $\delta \in \RR_+$ is a small universal quantity to be fixed later. Here we adopt the convention that $\arg(z) \in (-\pi, \pi]$ for all $z \in \mathbb{C} \setminus \{0\}$. Notice that
    \[\exp(df(r))=\frac{h(r)^d}{r^n}\]
    and
    \[
    f'(r)=\frac{h'(r)}{h(r)}- \frac{\alpha}{r}=0.
    \]

     Denote $\beta \coloneqq f^{(2)}(r)$. Simple computations show that
     \[
     f^{(2)}(z)=\frac{4z}{(1-z^2)^2}+\frac{\alpha}{z^2},
     \]
     hence
     \[
     \beta= \frac{4r}{(1-r^2)^2}+\frac{\alpha}{r^2}= \frac{4r}{(2r/\alpha)^2}+\frac{\alpha}{r^2}= \frac{\alpha^2}{r}+\frac{\alpha}{r^2}.
     \]
     In view of \eqref{eq: alpha}, the above gives us 
     \[
     \frac{\alpha}{r^2} \le \beta \le \alpha(\sqrt{2}+1)+ \frac{(\sqrt{2}+1)^2}{\alpha}.
     \]
   Consequently, we obtain
  \begin{equation}
\label{eq: beta}
   \frac{4}{\alpha}\le \beta \le \frac{9}{\alpha}.
   \end{equation}
   By complex Taylor's theorem for any $z \in \Gamma_r^{\leq \delta}$ we have 
    \begin{equation} \label{eq: taylor}
    f(z)=f(r)+ \frac{\beta}{2}(z-r)^2 + \frac{1}{2} \int_{\wideparen{r,z}} (w-z)^2 f^{(3)}(w) \, {\rm d}w,   \end{equation}
    where $\wideparen{r,z} \subseteq \Gamma_r^{\leq \delta}$ denotes the arc from $r$ to $z$. A computation shows that 
    \[
    f^{(3)}(w)= -\frac{2\alpha}{w^3} +\frac{4(1+3w^2)}{(1-w^2)^3}.
    \]
    Thus, by \eqref{eq: alpha}, for every $w \in \Gamma_r^{\leq \delta}$ we have
    \[
    |f^{(3)}(w)|\leq\frac{2\alpha}{r^3} +\frac{4(1+3r^2)}{(1-r^2)^3} \leq \frac{2(\sqrt{2}+1)^3 }{\alpha^2}+ \frac{7}{(3/4)^3} \le \frac{45}{\alpha^2}.
    \]
    If $\delta$ is sufficiently small, then combining the above bound with \eqref{eq: taylor} yields
   \begin{equation}
   \label{eq:ftay}
     |f(z)-f(r)- \frac{\beta}{2}(z-r)^2|\le \frac{45|z-r|^3}{\alpha^2}.
    \end{equation}

We are now ready to prove Theorem~\ref{thm:lat:extract}.  Some steps will involve taking $\delta$ ``small enough'' or $n$ ``large enough'' which means $\delta < c$ or $n > C$ for some universal constants $c,C\in \RR_+$. The number of these steps will be finite and so there will exist universal constants $c,C \in \RR_+$ such that all the statements in the proof hold for $\delta<c$ and $n>C$. Throughout the proof, $c$ with subscripts will denote nonnegative universal constants.
\begin{proof}[Proof of Theorem~\ref{thm:lat:extract}]
It suffices to consider the case $n \leq d$. If $n<C$, then 
    \[
\frac{h(r)^d}{r^n} \frac{1}{\sqrt{n}}\approx_C\left(1+\frac{2r}{1-r}\right)^{d}d^{n}\approx_C d^n
    \]
    in view of \eqref{eq: rdef} and \eqref{eq: alpha}. 
    Moreover, a simple combinatorial argument as in the proof  of \cite[Theorem~3.4]{NiWr} shows that 
    \[|B_n\cap \Z^d| \approx_C d^n\approx_C |S_n\cap \Z^d|.\]
Thus, it remains to consider $n>C.$ We begin by justifying \eqref{eq: B_n small univ}. For this purpose, we estimate $|B_n\cap \Z^d|$ from above and $|S_n\cap \Z^d|$ from below. 

\smallskip \noindent {\bf 1) Estimate from above for $|B_n\cap \Z^d|$.} 
  Recalling \eqref{eq:Cauchyball}, we have
    \begin{equation} \label{eq:3.3}
    |B_n\cap \Z^d|= \frac{1}{2 \pi i } \oint_{\Gamma_r} \frac{h(z)^d}{1-z} \frac{{\rm d}z}{z^{n+1}}
   =\frac{1}{2 \pi i } (W_1+W_2),
    \end{equation}
    where $W_1$ and $W_2$ denote the integrals along $\Gamma_r^{\leq \delta}$ and $\Gamma_r^{> \delta} \coloneqq \Gamma_r \setminus \Gamma_r^{\leq \delta}$, respectively, with some
    small $\delta \in (0,1/2)$ to be determined in the proof.
    
    We first estimate $W_1$. Using \eqref{eq: alpha}, \eqref{eq: beta}, \eqref{eq:ftay}, and the observation that
   \[
   \Rea ((1-e^{i\theta})^2)=1-2\cos \theta +\cos (2\theta)\le -\theta^2/2 
   \]
   holds whenever $|\theta| \leq 1/2$, 
   we obtain
  \begin{align*}
|W_1|&= \bigg| \int_{\Gamma_r^{\leq \delta}} \frac{\exp(df(z))}{1-z} \frac{{\rm d}z}{z} \bigg|
   \\ 
    &\le 
    \frac{h(r)^d}{r^{n}} \int_{-\delta}^{\delta} \exp\Big( \frac{d\beta r^2}{2} \Rea ((1-e^{i\theta})^2) + \frac{45d r^3}{\alpha^2} |1- e^{i\theta}|^3  \Big) \frac{{\rm d} \theta}{|1-re^{i \theta}|}
    \\
    &\lesssim \frac{h(r)^d}{r^{n}}  \int_{-\delta}^{\delta} \exp\Big( -\frac{d\beta r^2}{4}\theta^2+\frac{100d r^3}{\alpha^2} |\theta|^3 \Big)  \, {\rm d} \theta.
   \end{align*}
   Now, by \eqref{eq: alpha} and \eqref{eq: beta}, we see that
   \[
   \frac{d\beta r^2}{4}\approx d\alpha \approx \frac{100dr^3}{\alpha^2}.
   \]
   Taking $\delta$ small enough and substituting $\theta \mapsto x / \sqrt{d \beta r^2}$, we obtain 
    \begin{align*}
    |W_1| &\lesssim  \frac{h(r)^d}{r^{n}}\int_{-\delta}^\delta \exp\Big( - \frac{d\beta r^2}{8} \theta^2 \Big) \, {\rm d} \theta \lesssim \frac{h(r)^d}{r^{n}} \frac{1}{\sqrt{d \beta r^2}} \int_{-\infty}^\infty \exp(-x^2 / 8) \, {\rm d}x.
    \end{align*}
    This yields the desired estimate for $W_1$, since $d \beta r^2 \approx d \alpha = n$ by \eqref{eq: rdef}.  

\par Now we consider $W_2$. It is easy to see that if $z \in \Gamma_r^{> \delta}$, then
\begin{align*}
|h(z)|^2 & \leq |h(re^{i\delta})|^2 = \frac{1+r^2+2r\cos \delta}{1+r^2-2r\cos \delta}\le\frac{(1+r)^2+2r(\cos \delta-1)}{(1-r)^2}\\
&=h(r)^2 \bigg(1+\frac{2r(\cos\delta-1)}{(1+r)^2}\bigg)
\le h(r)^2 \exp \bigg( \frac{2r(\cos\delta-1)}{(1+r)^2} \bigg).
\end{align*}
Since $dr / (1+r)^2 \approx d\alpha=n$ by \eqref{eq: rdef} and \eqref{eq: alpha}, this implies
    \begin{equation}
\label{eq:W2bound}
   |W_2| \lesssim \frac{h(r)^d}{r^{n}} \exp \! \bigg( \! - \frac{dr(1 - \cos\delta)}{(1+r)^2} \bigg) \lesssim \frac{1}{\sqrt{n}}\frac{h(r)^d}{r^{n}}.
\end{equation}
 Recalling \eqref{eq:3.3}, we obtain the desired estimate
\begin{equation}
\label{eq:sB_nabo}
|B_n\cap \Z^d|\lesssim \frac{1}{\sqrt{n}}\frac{h(r)^d}{r^{n}}.
\end{equation}

\smallskip \noindent {\bf 2) Estimate from below for $|S_n\cap \Z^d|$.} As before, we split
\begin{equation*}
    |S_n\cap \Z^d|= \frac{1}{2 \pi i } \oint_{\Gamma_r} h(z)^d \frac{{\rm d}z}{z^{n+1}}
   =\frac{1}{2 \pi i } (V_1+V_2),
    \end{equation*}
    where $V_1$ and $V_2$ denote the integrals along $\Gamma_r^{\leq \delta}$ and $\Gamma_r^{> \delta}$, respectively, with some
    small $\delta \in (0,1/2)$ to be fixed during the proof. Regarding $V_2$, we have
\begin{equation}
\label{eq:V2ab}
|V_2|\lesssim  \exp (-c_3 n)\frac{h(r)^d}{r^{n}}   
\end{equation}
with some $c_3 \in \RR_+$ that depends only on $\delta$ (cf.~\eqref{eq:W2bound}). Regarding $V_1$, set
\[
\rho(\theta) \coloneqq \Rea (f(re^{i\theta})-f(r)),
\quad
\varphi(\theta) \coloneqq \Ima (f(re^{i\theta})-f(r)),
\]
where $f$ is defined in \eqref{eq:fdefi}.
Since $\exp(df(r))=h(r)^d r^{-n}$, we may write
\begin{equation}
\label{eq:ReaV1}
\Rea ( - i V_1) = \Rea \bigg(\int_{\Gamma_r^{\leq \delta}} e^{df(z)} \frac{{\rm d}z}{z}\bigg) = \frac{h(r)^d}{r^{n}} \int_{-\delta}^{\delta} e^{d \rho(\theta)} \cos(d\varphi(\theta)) \,{\rm d}\theta.
\end{equation}
Recalling \eqref{eq:ftay} and noting that $\Ima ((r-re^{i\theta})^2 )=O(r^2\theta^3)$, we obtain
\[
|d\varphi(\theta)|\lesssim d(\beta r^2 \theta^3 +r^3 \theta^3 \alpha^{-2})
\]
and the right-hand side if of order $n \theta^3$ by \eqref{eq: rdef}, \eqref{eq: alpha}, and \eqref{eq: beta}. If $n \geq C$ with $C$ large enough, then we may absorb the implicit constant and get 
\[
|d\varphi(\theta)|\le n^{-1/4}(\sqrt{n}\theta)^{3}
\]
for every $\theta \in [-\delta, \delta]$.
Moreover, if $\delta$ is small enough, then $\theta \in [-\delta, \delta]$ implies
\[
- \beta r^2 \theta^2 
\leq 
 \rho(\theta) \leq - \beta r^2 \theta^2 / 3
\]
by \eqref{eq:ftay} and the fact that $\lim_{\theta \to 0} \theta^{-2} \Rea((1-e^{i\theta})^2)  = -1$. Consequently, 
\[
-3 n \theta^2
\le d \rho(\theta) 
\leq - n\theta^2/9
\]
by \eqref{eq: rdef}, \eqref{eq: alpha}, and \eqref{eq: beta}. Substituting $\theta \mapsto s / \sqrt{n}$ in \eqref{eq:ReaV1}, we obtain
\[
\Rea ( - i V_1) =\frac{1}{\sqrt{n}} \frac{h(r)^d}{r^{n}}\int_{-\delta\sqrt{n}}^{\delta\sqrt{n}} 
e^{d \rho(s / \sqrt{n})} \cos(d\varphi(s / \sqrt{n}) \,{\rm d}s.
\]
We observe that $|d\varphi(s / \sqrt{n})|\le n^{-1/4} s^{3}$ and $-3s^2 \leq d\rho(s/\sqrt{n})\le -s^2/9$ hold if $|s|\le  \delta\sqrt{n}$. Let $C_* \in \RR_+$ be a fixed numerical constant such that
\[
c_* \coloneqq \frac{1}{2} \int_{-C_*}^{C_*} e^{-3s^2} {\rm d}s - \int_{-\infty}^{-C_*} e^{-s^2/9} \, {\rm d}s -
\int_{C_*}^{\infty} e^{-s^2/9} \, {\rm d}s
\]
is positive. Now, if $\delta\sqrt{n}>C_*$ and $n^{-1/4}C_*^3 \le \pi /3$, then necessarily
\begin{align*}
\int_{-\delta\sqrt{n}}^{\delta\sqrt{n}} 
e^{d \rho(s / \sqrt{n})} \cos(d\varphi(s / \sqrt{n}) \,{\rm d}s \geq c_* \gtrsim 1.
\end{align*}
Taking $C$ large enough depending on $\delta$, we conclude that for every $n \geq C$ the above estimate holds. In summary, we have shown that
\[
\Rea ( - i V_1) \gtrsim \frac{1}{\sqrt{n}}\frac{h(r)^d}{r^n} 
\]
holds and, combining this with \eqref{eq:V2ab}, we obtain the desired estimate
\begin{equation}
\label{eq:sS_nbel}
|S_n\cap \Z^d|\gtrsim \frac{1}{\sqrt{n}}\frac{h(r)^d}{r^n}.
\end{equation}
Since \eqref{eq:sB_nabo} and \eqref{eq:sS_nbel} together yield \eqref{eq: B_n small univ}, it remains to establish \eqref{eq: B_n small}.

\smallskip \noindent {\bf 3) Verification of \eqref{eq: B_n small}.} Recalling \eqref{eq: rdef}, we write
    \begin{equation}
    \label{eq:thm:3.4expformproof}
    \begin{split}
    \frac{h(r)^d}{r^n}
    &= \exp\big(d (\log(1+r)-\log(1-r)) - n\log(r) \big)
  \\
    &= (2/\alpha)^n \exp \big(n \alpha^{-1} (\log(1+r)-\log(1-r)) - n \log(2r/\alpha) \big)
    \end{split}
\end{equation}
and note that the map $\alpha \mapsto r = \frac{\sqrt{1+\alpha^2}-1}{\alpha}$ admits a holomorphic extension to the disk $\{z \in \CC : |z| < 1 \}$, which we still denote by $r$. Furthermore, 
\[
\sqrt{1+z^{2}}
= 1+ \sum_{k=1}^{\infty} \frac{\tfrac12 \big(\tfrac12-1\big)\cdots \big(\tfrac12-k+1\big)}{k!} \,z^{2k}
\]
holds inside the disk by the generalized binomial theorem. In particular, if $|z| \leq 3/4$, then Stirling's formula implies $|r(z)|\le 3/4$. Consequently, the map $\alpha \mapsto \alpha^{-1} (\log(1+r(\alpha))-\log(1-r(\alpha)))$ admits a holomorphic extension to the disk $\{ z \in \CC : |z| \le 3/4 \}$. Using the expansion above, we obtain
\[\frac{2r(z)}{z}=1+2\sum_{k=2}^{\infty} \frac{\tfrac12 \big(\tfrac12-1\big)\cdots \big(\tfrac12-k+1\big)}{k!}\, z^{2k-2}.
\]
Hence, if $|z|\le 2/3$, then we have 
\begin{align*}
\Big|\frac{2r(z)}{z}-1\Big|&\le \sum_{k=2}^{\infty} \bigg(\frac23\bigg)^{2k-2} = \frac45 <1,
\end{align*}
so that $\alpha \mapsto \log(2r(\alpha)/\alpha)$ also admits a holomorphic extension to the disk $\{z\in \CC : |z|<2/3\}.$ Summarizing all these observations, we conclude that 
\[
   \alpha^{-1} (\log(1+r)-\log(1-r)) - \log(2r/\alpha)=\sum_{k=0}^{\infty} b_k \alpha^{2k}
\]
for some coefficients $b_k$, with the series on the right-hand side above being absolutely convergent if $|\alpha| \le 1/2$ or, consequently, if $n\le d/2$. Finally, a straightforward but tedious calculation shows that 
\[
b_0=1,\quad b_1=\frac{1}{12}, \quad b_2=-\frac{3}{160}
\]
and, returning to \eqref{eq:thm:3.4expformproof}, 
we conclude that \eqref{eq: B_n small} is valid. 
\end{proof}

Using Theorem~\ref{thm:lat:extract}, we can easily derive Corollaries~\ref{cor:lat:compa small}~and~\ref{cor:lat:compa large}.
\begin{proof}[Proof of  Corollary~\ref{cor:lat:compa small}]
By \eqref{eq: B_n small} for $d,n \in \N$ satisfying $n \leq d/2$ we have
\begin{equation}
\label{eq: Bn Zd approx small}
|B_n\cap \Z^d|\approx |S_n\cap \Z^d| \approx (2e/\alpha)^n \frac{1}{\sqrt{n}
} \exp \Big(\frac{n \alpha^2}{12}+ O(n \alpha^{4}) \Big).
\end{equation}
Using Stirling's formula, we see that
\[
\binom{d}{n}= \frac{d(d-1) \cdots (d-n+1)}{n!} \approx (e/\alpha)^{n} \frac{1}{\sqrt{n}} \prod_{k=1}^{n-1} \Big(1- \frac{k}{d} \Big).
\]
We rewrite the logarithm of the product above as
\begin{align*}
\sum_{k=1}^{n-1} \log \Big(1- \frac{k}{d} \Big)
= - \sum_{k=1}^{n-1} \Big(\frac{k}{d} + \frac{k^2}{2d^2} \Big) + O \Big(\frac{n^4}{d^3} \Big)
= -\frac{n \alpha}{2} -\frac{n \alpha^2}{6}+ O(n \alpha^3+1)
\end{align*}
and, consequently, we obtain
\[
2^n \binom{d}{n}\approx  (2e/\alpha)^{n} 
\frac{1}{\sqrt{n}} \exp \Big( -\frac{n \alpha}{2} -\frac{n \alpha^2}{6}+ O(n \alpha^3) \Big). 
\]
Combining this estimate with \eqref{eq: Bn Zd approx small} completes the proof.
\end{proof}

\begin{proof}[Proof of Corollary~\ref{cor:lat:compa large}]
By \eqref{eq: B_n large} for $d,n \in \N$ satisfying $n \geq 2d$ we have
\[
|B_n\cap\Z^d|\approx (2e \alpha)^d \frac{1}{\sqrt{d}
}\exp \Big(  \frac{d}{12\alpha^2} + O(d/\alpha^{3}) \Big).
\]
Using Stirling's formula, we see that
\[
\Vol(B_n)= \frac{(2n)^d}{d!} \approx (2e\alpha)^d \frac{1}{\sqrt{d}}.
\]
Combining these two estimates completes the proof.
\end{proof}

\section{Full range of large scales -- Proof of Theorem~\ref{T1}}
\label{sec: full large}

Fix $c \in \RR_+$ and $p \in (1, \infty)$ as in the statement of Theorem~\ref{T1}. We take a large constant $C \in \RR_+$ to be specified later and depending only on $c$. For every fixed $d \in \NN$ the operator $M_{*,[cd^{3/2},\infty)}$ is bounded on $\ell^p(\ZZ^d)$, hence we can assume that $d \geq d_0$ for some large $d_0 \in \NN$ depending on $c, \, C, \, p$. Finally, we can replace $M_{*,[cd^{3/2},\infty)}$ by $M_{*,[cd^{3/2},\infty) \cap \ZZ}$.

From now on, we deal with integers $d \geq d_0$ and $n \geq cd^{3/2}$. We partition
\[
B_n \cap \ZZ^d = \bigcup_{l=0}^{d - \lfloor C d^{1/2} \rfloor} B_n^{l},
\]
where if $l > 0$, then $B_n^{l}$ consists of all $x = (x_1, \dots, x_d) \in B_n \cap \ZZ^d$ such that exactly $\lfloor C d^{1/2} \rfloor + l$ of the numbers $x_1, \dots, x_d$ are equal to $0$, while for each $x \in B_n^{0}$ at most $\lfloor C d^{1/2} \rfloor$ numbers are equal to $0$.

Let $Q \coloneqq [-\frac{1}{2}, \frac{1}{2}]$. The following lemma will be helpful.

\begin{lemma} \label{L1}
	There exists a constant $\delta = \delta(c,C) \in \RR_+$ such that the following is true. If
	$
	|x_1| + \dots + |x_d| \leq n
	$
	holds for some $x \in (B_n^{0} \cup B_n^{1}) \cap \ZZ^d$ with integers $d \geq d_0$ and $n \geq c d^{3/2}$,
	then
	\[
	\int_{Q^d} \int_{Q^d} \ind{B_n} (x+u+v) \, {\rm d}u {\rm d}v \geq \delta.
	\]
\end{lemma}

\begin{proof}
	For $x$ as above write $d = d' + d''$, where $d''$ counts how many of the numbers $x_1, \dots, x_d$ are equal to $0$. Of course, $d'' \leq \lfloor C d^{1/2} \rfloor + 1 \leq d/2$ by $d \geq d_0$ for $d_0$ large enough. By symmetry we can assume that $x_1, \dots, x_{d'}$ are positive and $x_{d'+1}, \dots, x_d$ are equal to $0$. By the central limit theorem, if $d_0$ is large enough, then for some constant $\delta = \delta(c,C) \in \RR_+$ we have
	\begin{equation} \label{L1-eq}
	    \mathbb P\big(U_1 + V_1 + \dots + U_{d_*} + V_{d_*} \leq - \lfloor \sqrt{2 C^2 d_*} \rfloor - 1\big) \geq \delta
	\end{equation}
	for all $d_* \in \NN$ satisfying $d_* \geq d_0 /2$, where $U_1, V_1, \dots, U_{d_*}, V_{d_*}$ are independent identically distributed random variables with density $\ind{[-1/2, 1/2]}$. Since $|x_i + u_i + v_i| = x_i + u_i + v_i$ for $i \in [d']$ 
    and $u_i,v_i \in Q$, we obtain
	\begin{equation} \label{E1}
	\int_{Q^{d'}} \int_{Q^{d'}} \ind{|x_1+u'_1+v'_1| + \dots + |x_{d'} + u'_{d'} + v'_{d'}| \leq x_1 + \dots + x_{d'} - d''} \, {\rm d}u' {\rm d}v' \geq \delta
	\end{equation}
	in view of $d' \geq d/2 \geq d_0/2$ and $\lfloor \sqrt{2C^2d'} \rfloor + 1 \geq \lfloor C d^{1/2} \rfloor + 1 \geq d''$. Using $|x_i + u_i + v_i| \leq 1$ for $i \in [d] \setminus [d']$ completes the proof. 
\end{proof}

For every $l \in \{0, 1, \dots, d - \lfloor C d^{1/2} \rfloor\}$ define an operator $M_n^{l}$ by setting
\[
M_n^{l} f(x) \coloneqq \frac{1}{|B_n \cap \ZZ^d|} \sum_{y \in B_n^{l}} f(x-y), \qquad x \in \ZZ^d,
\]
for all $f \colon \ZZ^d \to \RR$. We also consider the associated maximal function
\[
M_{*,[cd^{3/2},\infty) \cap \ZZ}^{l} f(x) \coloneqq
\sup_{n \in [cd^{3/2},\infty) \cap \ZZ} |M_n^{l} f(x)|, \qquad x \in \ZZ^d.
\] 
 
The following estimate holds for $l \in \{0,1\}$.

\begin{lemma} \label{L2}
	Let $l \in \{0,1\}$ and $p \in (1, \infty)$. Consider $\delta$ from Lemma~\ref{L1} and $\mathcal{C}(p)$ from \eqref{E0}. Then for an integer $d \geq d_0$ we have 
	\[
	\| M_{*,[cd^{3/2},\infty) \cap \ZZ}^{l} f \|_{\ell^p(\ZZ^d)} \lesssim \delta^{-1} \mathcal{C}(p) \| f \|_{\ell^p(\ZZ^d)}, \qquad f \in \ell^p(\ZZ^d).
	\]
\end{lemma}

\begin{proof}
	Fix $l, p, d$ as above and $f \in \ell^p(\ZZ^d)$. We define $F \colon \RR^d \to [0,\infty)$ by
	\[
	F(x) \coloneqq \sum_{y \in \ZZ^d} f(y) \ind{Q^d}(x-y).
	\]
	Note that $\| F \|_{L^p(\RR^d)} = \| f \|_{\ell^p(\ZZ^d)}$. Moreover, if $x \in \ZZ^d$, then we have
	\[
	M_n^{l} f(x) \lesssim \delta^{-1} \int_{Q^d} \mathcal M_n F(x+u) \, {\rm d}u
	\]
	for all integers $n \geq cd^{3/2}$ by Lemma~\ref{L1} and the bound $\Vol(B_n)\lesssim |B_n \cap \ZZ^d|$ from Corollary~\ref{cor:lat:compa large}. Taking the supremum over all such $n$, we obtain
	\[
	M_{*,[cd^{3/2}, \infty) \cap \ZZ}^{l} f(x) 
	\lesssim \delta^{-1} \Big( \int_{Q^d} \big( \mathcal M_{*} F(x+u) \big)^{p} \, {\rm d}u \Big)^{1/p}
	\]
    by H\"older's inequality.
	We complete the proof by raising the above expression to the $p$-th power, summing over $x \in \ZZ^d$, and using \eqref{E0}.
\end{proof}

It remains to deal with $2 \leq l \leq d - \lfloor C d^{1/2} \rfloor$.
\begin{lemma} \label{L3}
	If $C \geq 2 c^{-1}$, then 
	for integers $d \geq d_0$ and $n \geq c d^{3/2}$ we have
	\[
	|B^{1}_n \cap \ZZ^d| \geq 2 |B^{2}_n \cap \ZZ^d| 
	\geq 2^2 |B^{3}_n \cap \ZZ^d|
	\geq \cdots \geq 2^{d - \lfloor C d^{1/2} \rfloor - 1} |B^{d - \lfloor C d^{1/2} \rfloor}_n \cap \ZZ^d|. 
	\]
    In particular, for $l \in [d - \lfloor C d^{1/2} \rfloor]$ we have
    \begin{equation} \label{C1}
    \| M_{*,[cd^{3/2},\infty) \cap \ZZ}^{l} f \|_{\ell^\infty(\ZZ^d)} \leq 2^{1-l} \| f \|_{\ell^\infty(\ZZ^d)}, \qquad f \in \ell^\infty(\ZZ^d).   
    \end{equation}
\end{lemma}

\begin{proof}
 Fix $l \in [d - \lfloor C d^{1/2} \rfloor - 1]$ and set $l^* \coloneqq \lfloor C d^{1/2} \rfloor + l$. Then
 \[
 \frac{|B^{l+1}_n \cap \ZZ^d|}{|B^{l}_n \cap \ZZ^d|}
 = 
 \frac{\binom{d}{l^*+1} \binom{n}{d-l^*-1} 2^{d-l^*-1}}{\binom{d}{l^*} \binom{n}{d-l^*} 2^{d-l^*}}
 =
 \frac{(d-l^*)(d-l^*)}{2(l^*+1)(n-d+l^*+1)}
 \leq \frac{1}{2},
 \]
 since we have $l^*+1 \geq Cd^{1/2}$ and $n - d \ge \frac{1}{2}cd^{3/2}$ for $d_0$ large enough.
\end{proof}

We next aim to control the $\ell^{p}(\ZZ^d)$ norms of $M_{*,[cd^{3/2},\infty)}$ for $p \in (1,\infty)$.

\begin{lemma} \label{L4}
	Fix integers $d \geq d_0$, $n \geq c d^{3/2}$, and $l \in [d - \lfloor C d^{1/2} \rfloor]$. Consider $\delta$ from \eqref{L1-eq}. If
	$
	|x_1| + \dots + |x_d| \leq n
	$
	holds for some $x \in B_{n}^{l} \cap \ZZ^d$,
	then
	\[
	\int_{Q^d} \int_{Q^d} \ind{B_{n+2(l-1)}} (x+u+v) \, {\rm d}u {\rm d}v \geq \delta.
	\]
\end{lemma}

\begin{proof}
	The case $l=1$ holds by Lemma~\ref{L1}, hence we assume that $l > 1$. As in Lemma~\ref{L1}, write $d = d' + d''$ with $d'' \coloneqq \lfloor C d^{1/2} \rfloor + l$. By symmetry we can assume that $x_1, \dots, x_{d'}$ are positive and $x_{d'+1}, \dots, x_d$ are equal to $0$. Let $d^* \coloneqq d' + l - 1$. Then $|x_i + u_i + v_i| = x_i + u_i + v_i$ for $i \in [d']$ and $u_i,v_i \in Q$, while $|x_i + u_i + v_i| \leq x_i + u_i + v_i + 2$ for $i \in [d^*] \setminus [d']$. Thus,
	\[
	\int_{Q^{d^*}} \int_{Q^{d^*}} \ind{|x_1+u^*_1+v^*_1| + \dots + |x_{d^*} + u^*_{d^*} + v^*_{d^*}| \leq x_1 + \dots + x_{d^*} - d + d^* + 2(l-1)} \, {\rm d}u^* {\rm d}v^* \geq \delta
	\]
    by \eqref{L1-eq}, since $d^* \geq d/2 \geq d_0/2$ and $\lfloor \sqrt{2C^2d^*} \rfloor + 1 \geq  \lfloor C d^{1/2} \rfloor + 1 \geq d - d^*$ (cf.~\eqref{E1}). 
    Using $|x_i + u_i + v_i| \leq 1$ for $i \in [d] \setminus [d^*]$ completes the proof.
\end{proof}

Now, for each fixed $p \in (1, \infty)$ set $q \coloneqq \frac{1+p}{2}$
and choose a constant $C_p \in (1, \infty)$ such that $C_p^{q/p} 2^{-1+q/p} < 1$.
By interpolation with \eqref{C1}, it remains to prove the following estimate to establish Theorem~\ref{T1}.

\begin{lemma} \label{L5}
	Let $p \in (1,\infty)$. Consider $q, C_p$ as above, $\delta$ from \eqref{L1-eq}, and $\mathcal{C}(q)$ from \eqref{E0}. Then for integers $d \geq d_0$ and $l \in [ \lfloor C d^{1/2} \rfloor]$ we have
	\[
	\| M_{*,[cd^{3/2},\infty) \cap \ZZ}^{l} f \|_{\ell^{q}(\ZZ^d)} \lesssim C_p^{l-1}\delta^{-1} \mathcal{C}(q) \| f \|_{\ell^{q}(\ZZ^d)}, \qquad f \in \ell^{q}(\ZZ^d). 
	\]
\end{lemma}

\begin{proof}
	Fix $d, l$ as above and $f \in \ell^{q}(\ZZ^d)$. We follow the proof of Lemma~\ref{L2}, making only minor changes. If $x \in \ZZ^d$, then we have 
	\[
	M_n^{l} f(x) \lesssim \delta^{-1} \frac{\Vol(B_{n+2(l-1)})}{\Vol(B_n)} \int_{Q^d} \mathcal M_{n+2(l-1)}F(x+u) \, {\rm d}u
	\]
	for all integers $n \geq cd^{3/2}$ by Lemma~\ref{L4} and the bound $\Vol(B_R)\lesssim |B_R \cap \ZZ^d|$ from Corollary~\ref{cor:lat:compa large}.  
	If $d_0$ is large enough, then $d \geq d_0$ and $n \geq cd^{3/2}$ imply
	\[
	\frac{\Vol(B_{n+2(l-1)})}{\Vol(B_n)} = \Big( \frac{n+2(l-1)}{n} \Big)^d \leq \Big(1 + \frac{2(l-1)}{cd^{3/2}} \Big)^d \leq C_p^{l-1}.
	\]
	Hence, taking the supremum over all such $n$, we obtain
	\begin{align*}
	M_{*,[cd^{3/2}, \infty) \cap \ZZ}^{l} f(x)
	\lesssim C_p^{l-1} \delta^{-1} \Big( \int_{Q^d} \big( \mathcal M_{*} F(x+u) \big)^{q} \, {\rm d}u \Big)^{1/q}.
	\end{align*}
    by H\"older's inequality. 
	We complete the proof by raising the above expression to the $q$-th power, summing over $x \in \ZZ^d$, and using \eqref{E0}. 
\end{proof}

\begin{proof}[Proof of Theorem~\ref{T1}]
	We use Lemma~\ref{L2} and interpolate the estimates from Lemmas~\ref{L3}~and~\ref{L5}. 
\end{proof}

We conclude this section with the remark that the exponent $\frac{3}{2}$ is indeed the smallest possible for which our method works. This becomes evident when one attempts to reprove Lemma~\ref{L3} with any smaller exponent. 


\section{Full range of small scales -- proof of Theorem~\ref{T2}}

\label{sec: full small}

We now sketch how to prove Theorem~\ref{T2} using Theorem~\ref{thm:lat:extract} and some methods
from \cite{NiWr}. We begin with two simple consequences of Theorem~\ref{thm:lat:extract}.
\begin{corollary}
    \label{lem:3.8} Let $K \in \N$ and $\delta \in \RR_+$. There exists a constant $L_{K,\delta} \in \RR_+$ such that for all $d,n,m \in \N$ if $(1+\delta)m \leq n \leq d^{\frac{K}{K+1}}$ and $d \ge 2^{K+1}$, then 
    \[
    |B_{n-m}\cap \Z^d| \lesssim_{K,\delta} (L_{K,\delta} \alpha)^m |B_n\cap \Z^d|.
    \]
\end{corollary}
\begin{corollary}
\label{cor:3.9}   
Let $K,d,n \in \N$. If $10K \leq n \leq d^{\frac{K}{K+1}}$ and $d \ge 2^{K+1}$, then
\[
\big| \big \{x \in B_n\cap \Z^d : |x_i| \geq 6K \text{ for some } i \in [d] \big \} \big| \lesssim_{K} \frac{1}{d} |B_n\cap \Z^d|.
\]
\end{corollary}
\noindent The proofs of the corollaries are, \emph{mutatis mutandis}, the same as the proofs of \cite[Lemma~3.7]{NiWr} and \cite[Corollary~3.8]{NiWr} once we note that for each $K\in \N$ and $n\le d^{\frac{K}{K+1}} \le d/2$ we may write
\[
b(\alpha)=\sum_{k=1}^{\lfloor K/2 \rfloor} b_k \alpha^{2k}+O_K(\alpha^{K+1}).
\]

Using Corollaries~\ref{lem:3.8}~and~\ref{cor:3.9}, we can prove the following two results.

\begin{theorem}
    \label{thm:3.1}
    Let $K \in \N$ and $\varepsilon \in (0,1)$. There exists an integer $a_{K,\varepsilon} \in \N$ such that for all $d, n \in \N$ if $n \leq d^{\frac{K-\varepsilon}{K+1}}$, 
    then
    \begin{equation*} 
    \big| \big \{ x \in  B_n \cap \Z^d: \sum_{i \in [d]} |x_i| \, \ind{[-K,K]}(x_i) \leq n-a_{K,\varepsilon} \big \} \big| \lesssim_{K, \varepsilon} \frac{1}{d} |B_n \cap \Z^d|. 
    \end{equation*}
\end{theorem}
\begin{corollary} \label{cor:3.2}
Let $K \in \N$ and $\varepsilon \in (0,1)$. There exists an integer $a_{K,\varepsilon} \in \N$ such that for all $d, n \in \N$ if $n \leq d^{\frac{K-\varepsilon}{K+1}}$, 
then
    \begin{equation*} 
    \big | \big \{ x \in  S_n\cap \Z^d : \sum_{i \in [d]} |x_i| \, \ind{[-K,K]}(x_i) \leq n-a_{K,\varepsilon} \big \} \big | \lesssim_{K, \varepsilon} \frac{1}{d} |S_n\cap \Z^d|. 
    \end{equation*}
\end{corollary}

\noindent The proofs of Theorem~\ref{thm:3.1} and Corollary~\ref{cor:3.2} follow the lines of the proofs of \cite[Theorem~3.1]{NiWr} and \cite[Corollary~3.2]{NiWr}, respectively, and hence they are also omitted.
We expect both conclusions to fail for $\varepsilon=0$.

The next result is a consequence of Corollary~\ref{cor:lat:compa small} and \cite[Lemma~2.3]{Ni2}.
\begin{corollary} \label{cor:4.2}
 For all $d,n \in \N$ if $n/d=\alpha \le (818)^{-1}$, then
 \[
|\{x \in S_n\cap \Z^d: |\{i \in [d]: x_i= \pm 1\}| \leq n/2\}| \lesssim 2^{-n/2}|S_n\cap \Z^d|.
 \]   
\end{corollary}
\begin{proof}
By Corollary~\ref{cor:lat:compa small} it suffices to prove 
 \[
 |\{x \in B_n \cap \ZZ^d : |\{i \in [d]: x_i= \pm 1\}| \leq n/2\}| \lesssim 2^{-n/2}|B_n \cap \Z^d|.
 \]
 We then apply \cite[Lemma~2.3]{Ni2} with $q=1,$ $\alpha \le (818)^{-1}$, and $k=n/2$.
 \end{proof}

With Corollaries~\ref{cor:3.2}~and~\ref{cor:4.2} at hand, we prove Theorem~\ref{T2} by following \cite[Section~4]{NiWr}. We reduce the task to bounding the multiparameter combinatorial maximal function from \cite[Section~2]{NiWr}, which we now recall. 

Fix $K \in \N$ and a $K$-tuple $\vec j = (j_1, \dots, j_K)$ of nonnegative integers. Let $D_{\vec j}$ be the set of lattice points in $\{-K,\ldots,K\}^d$ such that exactly $j_1$ coordinates are equal to $\pm 1$, exactly $j_2$ are equal to $\pm 2$, and so on. Formally, we set
\begin{equation*}
D_{\vec j} \coloneqq \bigcap_{k \in [K]} \big\{ x \in \{-K,\dots,K\}^d :  |\{i \in [d]: |x_i|=k\}|=j_k \big\}.
\end{equation*}
Consider an averaging operator given by
\[
\mathcal{D}_{\vec j} f(x) \coloneqq \frac{1}{|D_{\vec j}|} \sum_{y \in D_{\vec j}} f(x-y), \qquad x \in \ZZ^d,
\]
and the corresponding multiplier symbol
\begin{equation*}
  \beta_{\vec j}(\xi) \coloneqq \frac{1}{|D_{\vec j}|} \sum_{x \in D_{\vec j}} e(x \cdot \xi),\qquad \xi\in\T^d.
\end{equation*}
Let $\vec J_{\leq d/2K} \coloneqq \{0, 1, \dots, \lfloor \frac{d}{2K} \rfloor \}^K$. In view of \cite[Theorem~1.5]{NiWr}, we have
\begin{equation}
\label{eq: combi max}
 \Big\| \sup_{\vec j \in \vec J_{\leq d/2K}} |\mathcal{D}_{\vec j}f| \Big\|_{\ell^2(\Z^d)} \lesssim_K \| f \|_{\ell^2(\Z^d)}, \qquad f\in\ell^2(\Z^d).
\end{equation}

Now, for $\xi = (\xi_1, \dots, \xi_d) \in \T^d$ let $\xi + 1/2 \coloneqq (\xi_1 + 1/2, \dots, \xi_d + 1/2) \in \T^d$ be its unique antipodal point. We partition $\T^d$ into two subsets
\[
\T^d_0 \coloneqq \{\xi \in \T^d: \|\xi \| \leq \| \xi + 1/2 \| \}, \quad
\T^d_1 \coloneqq \{\xi \in \T^d: \|\xi \| > \| \xi + 1/2 \| \},
\]
where $\| \xi \| \coloneqq \big( \sum_{i \in [d]} \sin^2(\pi \xi_i) \big)^{1/2}$. 
We shall prove the following result.
\begin{proposition} \label{prop Sn bounded by D}
    Let $K, d \in \N$ and $\varepsilon \in (0,1)$. If $d \geq (818+2K)^{K+1}$, then
   \[
   \Big\| \sup_{n \in [ \lfloor d^{\frac{K-\varepsilon}{K+1}} \rfloor ] } |\mathcal{S}_{n}f| \Big\|_{\ell^2(\Z^d)} \lesssim_{K,\varepsilon} \sum_{\omega \in \{0,1\}} \Big\| \sup_{\vec j \in \vec J_{\leq d/2K}} |\mathcal{D}_{\vec j}f_\omega| \Big\|_{\ell^2(\Z^d)} + \|f\|_{\ell^2(\Z^d)}
   \]
   for all $f \in \ell^2(\Z^d)$, where 
   $ f_\omega \coloneqq \mathcal{F}^{-1}( \ind{ \T^d_\omega} \widehat f) $ for $\omega \in \{0,1\}$.
\end{proposition} 

Since one has $M_{*,(0,d^{1-\varepsilon}]}f \le \mathcal S_{*,(0,d^{1-\varepsilon}]}f$ pointwise for nonnegative functions $f$, to prove Theorem~\ref{T2} it is enough to establish \eqref{eq: max spher small}. It is also not difficult to see that Proposition~\ref{prop Sn bounded by D}, together with \eqref{eq: combi max} and interpolation, gives \eqref{eq: max spher small} by taking $K$ large enough. 
It remains to prove Proposition~\ref{prop Sn bounded by D}.
\begin{proof}[Proof of Proposition~\ref{prop Sn bounded by D}]
We follow the proof of \cite[Proposition~4.3]{NiWr} and skip most details only commenting on necessary steps and changes. 

Fix $K, \varepsilon, d$ as above. The multiplier symbol of $\mathcal S_{n}$ is given by
\[
s_{n}(\xi) \coloneqq \frac{1}{|S_n\cap \Z^d|}\sum_{x\in S_n\cap \Z^d}e(x\cdot \xi),\qquad \xi \in \T^d.
\]  
For every $x \in \RR^d$ we decompose $x=y(x)+z(x)$ in such a way that 
\[
y(x)_i \coloneqq x_i \, \ind{[-K,K]}(x_i), \qquad i \in [d].
\]
Let $a_{K, \varepsilon} \in \NN$ be the integer from Corollary~\ref{cor:3.2}. We define the subset
\[
S_n^* \coloneqq \big \{ x \in S_n : \sum_{i \in [d]} |y(x)_i| \geq n - a_{K, \varepsilon} \text{ and }  |\{i \in [d] : |x_i|=1 \}| \geq n/2 \big \}
\]
and the corresponding multiplier symbol
\[
s^*_{n}(\xi) \coloneqq \frac{1}{|S_n\cap \Z^d|}\sum_{x\in S_n\cap \Z^d}e(x\cdot \xi) \, \ind{S_n^*}(x),\qquad \xi \in \T^d.
\]
By Corollaries~\ref{cor:3.2}~and~\ref{cor:4.2} if $ n \leq d^{\frac{K-\varepsilon}{K+1}}$, then 
\begin{equation}
\label{eq:rsnbound}
|s_{n}(\xi) - s^*_{n}(\xi)| \lesssim_{K,\varepsilon} \frac{1}{d} + 2^{-n/2}, \qquad \xi \in \T^d.
\end{equation}

The remaining term ${s}^*_{n}(\xi)$ may be treated as in \cite[Section~4]{NiWr}. Namely, denote $Z(S^*_n) \coloneqq \{z(x) \in \Z^d : x \in S^*_n\cap \Z^d\}$. For $l \in \{0, 1, \dots, a_{K, \varepsilon} \}$ let 
\[
Z_l(S^*_n) \coloneqq \{ z \in Z(S^*_n) : |\{i \in [d] : z_i \neq 0\}|=l \}.
\]
For $z \in Z(S^*_n)$ and a $K$-tuple $\vec j = (j_1, \dots, j_K)$ of nonnegative integers let
\[
Y_z(\vec j) \coloneqq \big \{ y \in D_{\vec j} : y_1 z_1 = \cdots = y_d z_d = 0 \text{ and } y+z \in S^*_n\cap \Z^d \big \}.
\]
Denoting by $\vec J_z$ the set of all admissible tuples, that is,
\[
\vec J_z \coloneqq \big \{ \vec j \in \{0,1,\dots,d\}^K : j_1 \geq n/2 \text{ and } n-a_{K, \varepsilon} \leq \sum_{k \in [K]} k j_k \leq n - |z|
\big\},
\]
where $|z| \coloneqq \sum_{i \in [d]} |z_i|$, we can rewrite $s^*_n(\xi)$ in the following way
\[
s^*_n(\xi) = \frac{1}{|S_n \cap \Z^d|}
\sum_{l=0}^{a_{K, \varepsilon}} \sum_{z \in Z_l(S^*_n)} e(z \cdot \xi)
\sum_{ \vec j \in \vec J_z} 
\sum_{y \in Y_{z}(\vec j)} e(y \cdot \xi).
\]
Note that $\sum_{i \in [d]} |z_i|^2 \le a_{K, \varepsilon}^2$ for $z \in Z(S^*_n)$.
Thus, for $\omega \in \{0,1\}$ we can get
\begin{align} 
\begin{split} \label{eq:4.1}
 |{s}^*_{n}(\xi)-\phi_{n,\omega}(\xi)| & \lesssim_{K,\varepsilon} \frac{1}{n}, \qquad \xi \in \T^d_\omega,
\end{split}
\end{align}
where, denoting $\sgn(0, z) \coloneqq 1$ and $\sgn(1, z) \coloneqq (-1)^{z_1+\dots+z_d}$, we define
\begin{align*}
\phi_{n,\omega}(\xi) & \coloneqq \frac{1}{|S_n \cap \Z^d|}
\sum_{l=0}^{a_{K, \varepsilon}} \sum_{z \in Z_l(S^*_n)} \sgn(\omega, z)
\sum_{ \vec j \in \vec J_z} 
\sum_{y \in Y_{z}(\vec j)} e(y \cdot \xi).
\end{align*}
The justification of \eqref{eq:4.1} is analogous to that in \cite[Section~4]{NiWr}. The product estimates for the multiplier symbol $\beta_{\vec j}$ in \cite[Lemma~2.8]{NiWr} are crucial here. 

Next, denoting $|\vec j| \coloneqq \sum_{k \in [K]} j_k$ and following \cite[Section~4]{NiWr}, we obtain 
\begin{align*}
\phi_{n,\omega}(\xi)
& = \frac{1}{|S_n \cap \Z^d|} \sum_{l=0}^{a_{K, \varepsilon}} 
\sum_{z \in Z_l(S^*_n)} \sgn(\omega, z)
\sum_{ \vec j \in \vec J_z}
\frac{\binom{d- |\vec j|}{l}}{\binom{d}{l}} |D_{\vec j}| \, \beta_{\vec j} (\xi).
\end{align*}
Note that $\phi_{n,0}(0)= {s}^*_{n}(0) \leq 1$. Thus, for $\omega \in \{0,1\}$ and $f \in \ell^2(\ZZ^d)$ we have 
\begin{equation}
\label{eq:4.4}
|\mathcal{F}^{-1} ( \phi_{n,\omega}\widehat{f} ) | 
\leq \sup_{\vec j \in \vec J_{\leq d/2K}} |\mathcal{D}_{\vec j} f|
\end{equation}
pointwise, using the fact that $d \geq (2K)^{K+1}$ and $n \leq d^{\frac{K-\varepsilon}{K+1}}$ imply $n \leq \frac{d}{2K}$.

Having established \eqref{eq:rsnbound}--\eqref{eq:4.4}, we complete the proof by repeating the final steps in the proof of \cite[Proposition~4.3]{NiWr}. Here the main ingredients are Plancherel's theorem, the square summability of $s_n(\xi) - \phi_{n,\omega}(\xi)$ in the range $n\le d/2$ for $\xi \in \T^d_\omega$, and the maximal estimate \eqref{eq: combi max}. 
\end{proof}

\section{Full range of dyadic scales -- proof of Theorem~\ref{T3}}
\label{sec: dyad}

The dimension-free bound 
\[
\| M_{*,\mathcal D\cap (0,d]} f \|_{\ell^p(\ZZ^d)} \leq C \| f \|_{\ell^p(\ZZ^d)},\qquad f\in \ell^p(\ZZ^d),
\]
for all $d \in \NN$ and $p \in [2,\infty)$ was recently established by the second author in \cite[Theorem~1.1]{Ni2}. Therefore, in view of Theorem \ref{T1}, it suffices to control the maximal function over $\mathcal D\cap (d,C_1d^{3/2})$ for some $C_1 \in (1,\infty)$. Moreover, since for each such $C_1$ the set $\mathcal D\cap (d,C_1 d]$ is finite and its size is uniformly bounded in $d$, it is enough to justify the following result.
\begin{theorem} \label{T3'}
	There exist universal constants $C_1,C_2 \in (1,\infty)$ such that
    \begin{equation*}
	\| M_{*,\mathcal D\cap (C_1d,C_1 d^{3/2})} f \|_{\ell^p(\ZZ^d)} \leq C_2 \| f \|_{\ell^p(\ZZ^d)},\qquad f\in \ell^p(\ZZ^d),
	\end{equation*}
    holds for all $d \in \NN$ and $p\in [2,\infty)$.
\end{theorem}
The new input from Theorem~\ref{thm:lat:extract} that will be used in the proof of Theorem~\ref{T3'} is contained in Corollary~\ref{cor: iso} below. This result provides an upper bound for a natural discrete analog of the isotropic constant for cross-polytopes (one can also obtain a lower bound of the same order). In fact, Corollary~\ref{cor: iso} can be generalized to a broad class of convex bodies with many symmetries but the argument is longer and less direct. This will be included in a forthcoming paper by the second author \cite{Ni3}.
\begin{corollary}
\label{cor: iso}
There exists a constant $C \in \RR_+$ such that if $n> Cd$, then
\[
\frac{1}{|B_n\cap \ZZ^d|}\sum_{x\in B_n\cap \ZZ^d}x_1^2 \lesssim \alpha^2.
\]
\end{corollary}
\begin{proof}
We may assume that $d > 1$, since the case $d=1$ is obvious. We have
\begin{align*}
    \frac{1}{|B_n\cap \ZZ^d|}\sum_{x\in B_n\cap \ZZ^d}x_1^2\leq \alpha^2 +  \sum_{j = \lfloor \alpha \rfloor}^n j^2 \frac{|A_j|}{|B_n\cap \ZZ^d|},
\end{align*}
where $A_j \coloneqq \{x\in B_n\cap \ZZ^d : |x_1|=j\}$.
Denoting by $B^*_{R}$ the closed $\ell^1$~ball in $\RR^{d-1}$ of radius $R$ (with $B^*_0$ being the singleton $\{0\}$ in $\RR^{d-1}$), we see that
\begin{equation}
\label{eq: Aj Bn}
\frac{|A_j|}{|B_n\cap \ZZ^d|}=2\frac{|B^*_{n-j}\cap \ZZ^{d-1}|}{|B_n\cap \ZZ^d|}.
\end{equation}  
By \eqref{eq: B_n large} from Theorem~\ref{thm:lat:extract} if $n \ge 2 d$ and $n-j\ge 2 (d-1)$, then
\begin{equation}
\label{eq: Vol com}
\begin{split}
 |B_n\cap \ZZ^d| &\approx (2e \alpha)^d \frac{1}{\sqrt{d}
}\exp (db(\alpha^{-1})),\\
|B^*_{n-j}\cap \ZZ^{d-1}|&\approx (2e \alpha_*)^{d-1}  \frac{1}{\sqrt{d}
}\exp (db(\alpha_*^{-1})),
\end{split}
\end{equation}
where $\alpha_* \coloneqq \frac{n-j}{d-1}$.
By Theorem~\ref{thm:lat:extract} and the mean value theorem we have \begin{align*}
|db(\alpha_*^{-1})-db(\alpha^{-1})| 
\leq L d |\alpha^{-1}-\alpha_*^{-1}|= Ld \frac{|dj-n|}{n(n-j)}
\le \frac{L d^2 j}{n(n-j)}+ L
\end{align*}
with $L \coloneqq \lfloor L(b) \rfloor + 1$, where $L(b)$ is the Lipschitz constant of $b$ on $[-\frac12,\frac12]$.
We consider the range $n \ge 20Ld$. If $n-j\ge  2Ld$, then $\frac{Ld^2 j}{n(n-j)} \leq \frac{j}{2\alpha}$. Thus, 
\begin{align*}
    \frac{|A_j|}{|B_n\cap \ZZ^d|} 
    &\lesssim\alpha^{-1} (1-j/n)^{d-1} \exp(db(\alpha_*^{-1})-b(\alpha^{-1}))\\
    &\lesssim \alpha^{-1}\exp(-j/\alpha)\exp (j/(2\alpha))\le \alpha^{-1}\exp(-j/(2\alpha))
\end{align*}
by \eqref{eq: Aj Bn} and \eqref{eq: Vol com}. Noting that $\lfloor \alpha \rfloor \leq n - 2Ld$ by $n \ge 20Ld$, we see that
\begin{equation*}
\begin{split}
 \sum_{ j = \lfloor \alpha \rfloor }^{n - 2Ld} j^2 \frac{|A_j|}{|B_n\cap \ZZ^d|} \lesssim \sum_{j \in \NN} j^2 \alpha^{-1} \exp(-j/(2\alpha))
 \lesssim \sum_{j \in \NN} \alpha \exp(-j/(4\alpha))\lesssim \alpha^2.
 \end{split}
\end{equation*}
On the other hand, if $n-j\le 2Ld$, then \eqref{eq: Aj Bn} and \eqref{eq: Vol com} imply
\begin{align*}
    \frac{|A_j|}{|B_n\cap \ZZ^d|} &\lesssim \frac{|B_{2Ld}^* \cap \ZZ^{d-1}|}{|B_n\cap \ZZ^d|}\lesssim (2Ld/n)^{d}  \exp(d/2) \lesssim \frac{d}{n}\exp(-d/2),
\end{align*}
where the last inequality holds by $(2Ld/n)^{d-1} \leq 10^{-d/2} \leq e^{-d}$. Thus,
\begin{equation*}
 \sum_{j = n-2Ld}^{n} j^2 \frac{|A_j|}{|B_n\cap \ZZ^d|}\lesssim \frac{d}{n}\exp(-d/2) \sum_{j \in [n]} j^2 \lesssim \alpha^2 d^3 \exp(-d/2) \lesssim \alpha^2
\end{equation*}
and the proof is complete.
\end{proof}

We are now in a position to prove Theorem~\ref{T3'}.

\begin{proof}[Proof of  Theorem~\ref{T3'}]
We briefly sketch how to adapt the proof of \cite[Theorem~3]{KMPW}. Consider the multiplier symbol corresponding to $M_n$, that is, 
\[
m_n(\xi) \coloneqq \frac{1}{|B_n\cap \ZZ^d|}\sum_{x \in B_n\cap \ZZ^d}e( x\cdot \xi), \qquad \xi\in \RR^d.
\]
Let $C_1 \coloneqq 10+ C$ with $C$ from Corollary~\ref{cor: iso}. Now, the main ingredients are \cite[Proposition~4.1]{KMPW} and \cite[Proposition~4.2]{KMPW} which in our cases become
\begin{equation}
\label{eq: mul loc}
|m_n(\xi)-1|\lesssim (\alpha|\xi|)^2
\end{equation}
and
\begin{equation}
\label{eq: mul glob}
|m_n(\xi)|\lesssim (\alpha |\xi|)^{-1}+\alpha^{-1/7}.
\end{equation}

The proof of \eqref{eq: mul glob} follows the lines of the proof of \cite[Proposition~4.2]{KMPW}. The repetition of the argument is possible once we note that we can take any $q\in [1,\infty)$ and only need the same lower bound $\alpha=\kappa_1(d,n)\ge 10$ and the weaker upper bound $\alpha=\kappa_1(d,n) \le C_1 d^{1/2}$, which hold if $n \in (C_1d,C_1d^{3/2})$.

The proof of \eqref{eq: mul loc} is based on Corollary~\ref{cor: iso}. Proceeding as in the proof of \cite[Proposition~4.1]{KMPW} and using symmetry and the corollary, we obtain
\begin{align*}
|m_n(\xi)-1|&\le \frac{2}{|B_n \cap \ZZ^d|} \sum_{j \in [d]} \sin^2(\pi \xi_j) \sum_{x\in B_n \cap \ZZ^d} \frac{x_1^2 + \dots + x_d^2}{d} \\
&=\frac{2}{|B_n \cap \ZZ^d|} \sum_{j \in [d]} \sin^2(\pi \xi_j) \sum_{x\in B_n \cap \ZZ^d} x_1^2 \lesssim (\alpha|\xi|)^2.
\end{align*}

Finally, we apply \eqref{eq: mul loc} and \eqref{eq: mul glob} as in the proof of \cite[Theorem~3]{KMPW}.
\end{proof}

\renewcommand{\baselinestretch}{1.08}

\end{document}